\documentclass[12pt,leqno]{amsart}
\usepackage{amssymb}
\usepackage{amsmath,amssymb,color}
\textheight 
18.5cm
\textwidth 
12.5cm

\oddsidemargin    
2.21cm 
\evensidemargin 
 2.21cm   
\newtheorem{thm}{Theorem}[section]
\newtheorem{lem}[thm]{Lemma}
\newtheorem{cor}[thm]{Corollary}

\newtheorem{rmk}{Remark}

\numberwithin{equation}{section}

\newcommand{\ep}{\varepsilon}
\newcommand{\la}{\lambda}
\newcommand{\va}{\varphi}
\newcommand{\ppp}{\partial}

\newcommand{\hhhh}{H^2(\Omega) \cap H^1_0(\Omega)}

\newcommand{\pppa}{\partial_t^{\alpha}}

\newcommand{\sumij}{\sum_{i,j=1}^d}
\newcommand{\sumn}{\sum_{n=1}^{\infty}}
\newcommand{\eeaa}{E_{\alpha,1}(-\lambda_nt^{\alpha})}

\newcommand{\DDDD}{\mathcal{D}}

\newcommand{\R}{\mathbb{R}}

\newcommand{\C}{\mathbb{C}} 
\newcommand{\N}{\mathbb{N}}

\newcommand{\ooo}{\overline}
\newcommand{\OOO}{\Omega}

\allowdisplaybreaks

\title[]{
\vspace{0.5cm}
Well-posedness for the backward problems in time for general time-fractional 
diffusion equation}

\author[Giuseppe Floridia,
Zhiyuan Li,
Masahiro Yamamoto]{
$^{1,*}$ Giuseppe Floridia,
$^2$ Zhiyuan Li,
$^{3,4,5}$ Masahiro Yamamoto }

\thanks{\footnotesize{
\noindent$^{1}$
Department PAU, 
Universit\`a Mediterranea di Reggio Calabria,
Via dell'Universit\`a 25  
89124 Reggio Calabria, Italy \&
INdAM Unit, University of Catania, Italy,
{\tt floridia.giuseppe@icloud.com},\;
$^{*}${Corresponding author}\\
$^2$ 
School of Mathematics and Statistics, Shandong University of Technology, 
Zibo, Shandong 255049, People's Republic of China,
{\tt zyli@sdut.edu.cn}\\
$^3$ Graduate School of Mathematical Sciences, The University
of Tokyo, 
Komaba, Meguro, Tokyo 153-8914, Japan \\
$^4$ Honorary Member of Academy of Romanian Scientists, 
Splaiul Independentei Street, no 54,
050094 Bucharest Romania \\
$^5$ Peoples' Friendship University of Russia 
(RUDN University) 6 Miklukho-Maklaya St, Moscow, 117198, Russian Federation,
{\tt myama@ms.u-tokyo.ac.jp}\;\,.
}}

\date{}
\begin{document}
\maketitle

\baselineskip 18pt

\begin{abstract}
\footnotesize{\smaller\smaller
In this article, we consider an evolution partial differential equation 
with Caputo time-derivative with the zero Dirichlet boundary condition: 
$\pppa u + Au = F$ where $0< \alpha < 1$ and 
the principal part $-A$, is a non-symmetric elliptic operator of 
the second order.
Given a source F, we prove the well-posedness for the 
backward problem in time 
and our result generalizes the existing results assuming that 
$-A$ is symmetric.
The key is a perturbation argument and the completeness of the 
ge\-neralized eigenfunctions of the elliptic operator $A$.\\
{\bf Key words:}  
fractional PDE, backward problem,
well-posedness
\\
{\bf AMS subject classifications:}
35R11, 34A12\,.}
\end{abstract}

\section{Introduction and main results}

Let $\OOO$ be a bounded domain in $\R^d$ with sufficiently smooth 
boundary $\ppp\OOO$.  
Henceforth let $L^2(\OOO)$ denote the real Lebesgue space with the scalar 
product $(\cdot, \cdot)$ and the norm $\Vert \cdot\Vert$, and let 
$H^1(\OOO), H^1_0(\OOO), H^2(\OOO)$ be the Sobolev spaces
(e.g., Adams \cite{Ad}).  By $\Vert u\Vert_{H^2(\OOO)}$ we denote the norm in 
$H^2(\OOO)$ for example.

We consider a fractional partial differential equation:
\begin{equation}\label{(1.1)}
\left\{ \begin{array}{rl}
&\pppa u(x,t) = -Au(x,t) + F(x,t), \quad x \in \OOO, \,0<t<T, \\
& u\vert_{\ppp\OOO} = 0, \\
& u(x,0) = a(x), \quad x \in \OOO.                
\end{array}\right.
\end{equation}
Here $-A$ is a uniformly elliptic operator and not necessarily 
symmetric.   Throughout this article, we assume that $0 < \alpha < 1$, and the 
Caputo derivative $\pppa g$ is defined by 
$$
\pppa g(t) = \frac{1}{\Gamma(1-\alpha)} \int^t_0 (t-s)^{-\alpha}
\frac{dg}{ds}(s) ds, 
$$
where $\Gamma$ denotes the gamma function.  It is known that 
there exists a unique solution $u=u(x,t)$ to 
the initial boundary value problem \eqref{(1.1)} under suitable conditions on 
$A$, $a$ and $F$, and we refer for example to Gorenflo, Luchko and 
Yamamoto \cite{GLY}, Kubica, Ryszewska and Yamamoto \cite{KRY},
Kubica and Yamamoto \cite{KY}, Sakamoto and Yamamoto \cite{SY},
Zacher \cite{Za}, and also later as lemmata we will show the regularity.

Equation \eqref{(1.1)} describes slow diffusion which can be considered as 
anomalous diffusion in highly heterogeneous media and is different 
from the classical case of $\alpha=1$.
In particular, the Caputo derivative is involved with memory term 
which possesses some averaging effect, and so 
\eqref{(1.1)} has not strong smoothing property: for $a\in L^2(\OOO)$, we can
expect only $u(\cdot,t) \in H^2(\OOO)$ with each $t>0$.
This is an essential difference from the case of $\alpha=1$.

Now we will formulate our problem and results.
For $v\in H^2(\OOO),$ we set 
\begin{equation}\label{(1.2)}
-Av(x) := \sumij \ppp_i(a_{ij}(x)\ppp_jv)(x)
+ \sum_{j=1}^d b_j(x)\ppp_jv(x) + c(x)v(x), 
\end{equation}
where
$$
a_{ij} = a_{ji} \in C^1(\ooo{\OOO}), \quad b_j, c \in C^1(\ooo{\OOO}),
\quad 1\le i,j \le d
$$
and there exists a constant $\kappa>0$ such that 
$$
\sumij a_{ij}(x)\xi_i\xi_j \ge \kappa \sum_{j=1}^d \xi_j^2,
\quad x \in \ooo{\OOO}, \, \xi_1, ..., \xi_d \in \R.
$$
We consider 
\begin{equation}\label{(1.3)}
\left\{ \begin{array}{rl}
& \pppa u(x,t) = -Au(x,t), \quad x\in \OOO, \, 0<t<T, \\
& u\vert_{\ppp\OOO} = 0, \\
& u(\cdot,T) = b
\end{array}\right.                      
\end{equation}
with $b \in H^2(\OOO) \cap H^1_0(\OOO)$.

We state our first main result.
\begin{thm}\label{Theorem 1}
For each $b\in H^2(\OOO)\cap H^1_0(\OOO)$, there exists a unique 
solution $u \in C([0,T];L^2(\OOO)) \cap C((0,T];\hhhh)$ to \eqref{(1.3)} such that 
$\pppa u \in C((0,T];L^2(\OOO))$.  Moreover we can choose constants
$C_1, C_2 > 0$ depending on $T$ such that 
\begin{equation}\label{(1.4)}
C_1\Vert u(\cdot,0)\Vert_{L^2(\OOO)} 
\le \Vert u(\cdot,T)\Vert_{H^2(\OOO)} 
\le C_2\Vert u(\cdot,0)\Vert_{L^2(\OOO)}.    
\end{equation}
\end{thm}
 
To the best knowledge of the authors, Sakamoto and Yamamoto \cite{SY}
is the first work for the well-posedness of the backward problem in time
for the case of symmetric $A$, that is, $b_j\equiv 0$ for $1\le j \le d$.
Moreover by a technical reason, \cite{SY} assumes that $c\le 0$.   
As for backward problems for time-fractional equations with 
symmetric $A$, we can refer to many works: Liu and Yamamoto \cite{LY},
Tuan, Huynh, Ngoc, and Zhou \cite{THNZ}. 
In particular, as for numerical approaches, see Tuan, Long and Tatar 
\cite{TLT}, Tuan, Thach, O'Regan, and Can \cite{TNOZ}.
Wang and Liu \cite{WL1, WL2},
Wang, Wei and Zhou \cite{WWZ}, Wei and Wang \cite{WW}, Xiong, Wang and Li 
\cite{XWL}, Yang and Liu \cite{YL} and the references therein.
However, we do not find the results for non-symmetric $A$.
Originally the backward well-posedness comes from the time fractional 
derivative 
$\pppa$, and should not rely on the symmetry of the elliptic operator
$A$, and Theorem \ref{Theorem 1} is a natural generalization of the existing results 
since \cite{SY} to the case of a general uniform elliptic operator $A$.
As is seen by the proof, we can further prove\\
\begin{cor}
In Theorem \ref{Theorem 1}, for each distinct $T_1, T_2 > 0$, there exist
contants $C_3=C_3(T_1,T_2)>0$ and $C_4=C_4(T_1,T_2)>0$ such that 
$$
C_3\Vert u(\cdot,T_2)\Vert_{H^2(\OOO)} 
\le \Vert u(\cdot,T_1)\Vert_{H^2(\OOO)} 
\le C_4\Vert u(\cdot,T_2)\Vert_{H^2(\OOO)}. 
$$
\end{cor}

Furthermore we can show also the backward well-posedness with the 
presence of a non-homogeneous term $F$.  \\
For the formulation, 
we introduce some function spaces.
Let 
\begin{equation}\label{(1.5)}
-A_0v(x) = \sumij \ppp_i(a_{ij}(x)\ppp_jv), \quad 
\mathcal{D}(A_0) = \hhhh.                
\end{equation}
Then it is known that the specrum $\sigma(A_0)$ consists entirely of 
eigenvalues with finite multiplicities and according to 
the multiplicities we number:
\begin{equation}\label{(1.6)}
0 < \la_1 \le \la_2 \le \la_3 < \cdots .        
\end{equation}
Also we know that we can choose eigenfunctions $\va_n$ 
for $\la_n$, $n \in \N$ such that $\{\va_n\}_{n\in \N}$ is an orthonormal
basis in $L^2(\OOO)$.
Then we can define the fractional power $A_0^{\gamma}$ with 
$\gamma\ge 0$:
\begin{equation}\label{(1.7)}
\left\{ \begin{array}{rl}
& A_0^{\gamma}v = \sum_{n=1}^{\infty} \lambda_n^{\gamma}(v, \va_n)\va_n,\cr\\
&\mathcal{D}(A_0^{\gamma}) = \left\{ v\in L^2(\OOO);\, 
\sum_{n=1}^{\infty} \lambda_n^{2\gamma}\vert (v, \va_n)\vert^2 < \infty
\right\}, \cr\\
& \Vert A_0^{\gamma}v\Vert = \left(
\sum_{n=1}^{\infty} \lambda_n^{2\gamma}\vert (v, \va_n)\vert^2\right)
^{\frac{1}{2}}. \cr
\end{array}\right.
\end{equation}
We can refer for example to Pazy \cite{Pa} and we can derive (1.7) directly 
from 
\begin{align*}
& A_0v = \sum_{n=1}^{\infty} \lambda_n(v, \va_n)\va_n,\\
&\mathcal{D}(A_0) = \left\{ v\in L^2(\OOO);\, 
\sum_{n=1}^{\infty} \lambda_n^2\vert (v, \va_n)\vert^2 < \infty
\right\}.
\end{align*}
Moreover we know that 
$\mathcal{D}(A_0^{\frac{1}{2}}) = H^1_0(\OOO)$,
$\DDDD(A_0^{\gamma}) \subset H^{2\gamma}(\OOO)$.
Henceforth we set $\Vert v\Vert_{\DDDD(A_0^{\gamma})}
= \Vert A_0^{\gamma}v\Vert$.  

Now we are ready to state the well-posedness with non-homogeneous term.

\begin{thm}\label{Theorem 2}
Let $F \in L^{\infty}(0,T;\DDDD(A_0^{\ep}))$ with some 
$\ep > 0$.  For each $b \in \hhhh$, there exists a unique solution
$$u \in C((0,T]; \hhhh) \cap C([0,T];L^2(\OOO))$$ to 
$$
\left\{ \begin{array}{rl}
& \pppa u = -Au + F(x,t), \quad x \in \OOO, \, 0<t<T, \\
& u\vert_{\ppp\OOO} = 0, \\
& u(\cdot,T) = b
\end{array}\right.                   
$$
and we can choose a constant $C>0$ such that 
$$
\Vert u(\cdot,0)\Vert 
\le C(\Vert u(\cdot,T)\Vert_{H^2(\OOO)}
+ \Vert F\Vert_{L^{\infty}(0,T;\DDDD(A_0^{\ep}))}).
$$
\end{thm}

The article is composed of three sections.  In Section \ref{sec2}, we show
fundamental properties of the fractional differential equations and 
Section \ref{sec3} is devoted to the proofs of Theorems \ref{Theorem 1} and \ref{Theorem 2}.
\section{Preliminaries}\label{sec2}

Let us recall \eqref{(1.5)} and \eqref{(1.6)}.  For $0 < \alpha < 1$ and 
$\beta > 0$, by $E_{\alpha,\beta}(z)$ we denote the 
Mittag-Leffler function with two parameters:
$$
E_{\alpha,\beta}(z) = \sum_{k=0}^{\infty} \frac{z^k}{\Gamma(\alpha k + \beta)}
$$
(e.g., Podlubny \cite{Po}).
Then $E_{\alpha,\beta}(z)$ is an entire function in $z\in \C$.  We set
$$
S(t)a = \sum_{n=0}^{\infty} (a,\va_n)E_{\alpha,1}(-\la_nt^{\alpha})
\va_n(x), \quad t\ge 0
$$
and
$$
K(t)a = \sum_{n=0}^{\infty} t^{\alpha-1}E_{\alpha,\alpha}(-\la_nt^{\alpha})
(a,\va_n)\va_n(x), \quad t>0
$$
for $a \in L^2(\OOO)$.
\\

Henceforth we write $u(t) = u(\cdot,t)$, etc., and we regard $u$ as
a mapping defined in $(0,T)$ with values in $L^2(\OOO)$.
Moreover $u(t) \in H^1_0(\OOO)$ means $u(\cdot,t) = 0$ on 
$\ppp\OOO$ in the trace sense (e.g., \cite{Ad}).
Then we can see the following.

\begin{lem}\label{Lemma 1}
\begin{enumerate}
\item[(i)] There exists a constant $C>0$ such that 
\begin{equation}\label{(2.1)}
\Vert S(t)a\Vert\le C\Vert a\Vert, \quad t\ge 0
\end{equation}
and
\begin{equation}\label{(2.2)}
\Vert A_0S(t)a\Vert\le Ct^{-\alpha}\Vert a\Vert, \quad t>0.
\end{equation}
For $0 \le \gamma \le 1$, there exists a constant $C(\gamma)>0$ such that 
\begin{equation}\label{(2.3)}
\Vert A_0^{\gamma}K(t)a\Vert\le C(\gamma)t^{\alpha(1-\gamma)-1}
\Vert a\Vert, \quad t > 0.
\end{equation}
\item[(ii)] Let $G \in L^{\infty}(0,T;\DDDD(A_0^{\ep}))$ with some
$\ep>0$ and $a \in L^2(\OOO)$.  Then
\begin{equation}\label{(2.4)}
u(t) = S(t)a + \int^t_0 K(t-s) G(s) ds, \quad t>0    
\end{equation}
is in $C((0,T];\hhhh)$ and satisfies $\pppa u \in L^1(0,T;L^2(\OOO))$,
\begin{equation}\label{(2.5)}
\left\{ \begin{array}{rl}
& \pppa u(t) = -A_0u(t) + G(t), \quad t>0, \\
& \lim_{t\to 0} \Vert u(\cdot,t) - a\Vert = 0, \\
& u(\cdot,t) \in H^1_0(\OOO), \quad  0 < t < T.
\end{array}\right.
\end{equation}
\item[(iii)] For each $t>0$, there exists a constant $C>0$ such that 
$$
\Vert u(t)\Vert_{H^2(\OOO)} \le C(t^{-\alpha}\Vert a\Vert 
+ \Vert A_0^{\ep}G\Vert_{L^{\infty}(0,T;L^2(\OOO))}).
$$
\end{enumerate}
\end{lem}

\begin{rmk}
We can prove stronger regularity of $\pppa u$ but the lemma is
sufficient for our purpose.
\end{rmk}
\begin{proof}{(of Lemma \ref{Lemma 1}).}\\
(i) We can refer to Gorenflo, Luchko and Yamamoto \cite{GLY}, and for
completeness we give the proof.  First we note 
\begin{equation}\label{(2.6)}
\vert E_{\alpha,1}(-\eta)\vert \le \frac{C}{1 + \eta},
\quad \eta > 0              
\end{equation}
(e.g., Theorem 1.6 (p.35) in Podlubny \cite{Po}).

Since $\{\va_n\}_{n\in \N}$ is an orthonormal basis in $L^2(\OOO)$, by \eqref{(2.6)} 
we have
\begin{multline*}
\Vert S(t)a\Vert^2 = \sumn \vert (a,\va_n)\vert^2
\vert \eeaa\vert^2\\
\le \sumn \vert (a,\va_n)\vert^2 \left(\frac{C}{1+\vert \la_nt^{\alpha}\vert}
\right)^2
\le C\sumn \vert (a,\va_n)\vert^2,
\end{multline*}
that is, \eqref{(2.1)} follows.

Next, since 
$$
A_0S(t)a = \sumn (a,\va_n) \la_n \eeaa \va_n,
$$
again by \eqref{(2.6)} we see
\begin{align*}
& \Vert A_0S(t)a\Vert^2 
= t^{-2\alpha}\sumn \vert (a,\va_n)\vert^2 \vert \la_nt^{\alpha}\vert^2
\vert \eeaa\vert^2\\
\le & Ct^{-2\alpha} \sumn \vert (a,\va_n)\vert^2 
\left( \frac{\vert \la_nt^{\alpha}\vert}{1+\vert \la_nt^{\alpha}\vert}
\right)^2, \quad t>0,
\end{align*}
which implies \eqref{(2.2)}.

By (1.7), we have
$$
A_0^{\gamma}K(t)a
= \sumn t^{\alpha-1}E_{\alpha,\alpha}(-\la_nt^{\alpha})
\la_n^{\gamma}(a,\va_n)\va_n,
$$
and so 
\begin{align*}
& \Vert A_0^{\gamma}K(t)a\Vert^2
\le t^{2\alpha-2}\sumn \frac{C}{(1+\vert \la_nt^{\alpha}\vert)^2}
\la_n^{2\gamma} \vert (a,\va_n)\vert^2\\
= &Ct^{2\alpha-2}\sumn \frac{\la_n^{2\gamma}t^{2\gamma\alpha}}
{(1+\vert \la_nt^{\alpha}\vert)^2}t^{-2\alpha\gamma}\vert (a,\va_n)\vert^2\\
\le &Ct^{2(\alpha-\alpha\gamma)-2}
\sup_{\xi\ge 0} \left( \frac{\xi^{\gamma}}{1+\xi} \right)^2
\sumn \vert (a,\va_n)\vert^2.
\end{align*}
By $0 \le \gamma \le 1$, we see that 
$\sup_{\xi\ge 0} \frac{\xi^{\gamma}}{1+\xi} < \infty$, and so \eqref{(2.3)}
can be seen.  Thus the proof of Lemma \ref{Lemma 1} (i) is complete.
\\
(ii) In terms of e.g., Theorem 4.1 in \cite{GLY} and Theorems 2.1 and 2.2 in 
\cite{SY}, we already know some regularity of $u(t)$.

By Theorem 2.1 (i) in \cite{SY} or by \eqref{(2.1)}, we can verify that
$S(t)a \in C([0,T];L^2(\OOO))$ and $\displaystyle\lim_{t\to 0} \Vert S(t)a - a \Vert = 0$.  
By \eqref{(2.2)}, we see that 
$$
A_0\left( \sum_{n=1}^N (a,\va_n)\eeaa \va_n\right)
$$
converges in $C([\delta,T];L^2(\OOO))$ as $N \to \infty$ with arbitrarily 
fixed $\delta>0$.  Therefore $A_0S(t)a \in C([\delta,T];L^2(\OOO))$, which 
implies 
\begin{equation}\label{(2.7)}
S(t)a \in C([\delta,T];\DDDD(A_0)) = C([\delta,T];\hhhh).  
\end{equation}

Moreover, we can directly prove that 
$\pppa (\eeaa) = -\la_n \eeaa$, and obtain
$$
\pppa S(t)a = \sumn \pppa (\eeaa) (a,\va_n)\va_n
= \sumn -\la_n \eeaa (a,\va_n)\va_n.
$$
Hence, by (2.6) we see that 
\begin{multline}\label{(2.8)}
\Vert \pppa S(t)a\Vert^2
= \sumn \la_n^2 \vert \eeaa \vert^2 \vert (a,\va_n)\vert^2\\
= t^{-2\alpha}\sumn (\la_n t^{\alpha})^2
\vert \eeaa\vert^2 \vert (a,\va_n)\vert^2\\
\le Ct^{-2\alpha}\sumn \vert (a,\va_n)\vert^2 \left(
\frac{\la_nt^{\alpha}}{1+\la_n t^{\alpha}}\right)^2
\le Ct^{-2\alpha}\Vert a\Vert^2               
\end{multline}
and
\begin{equation}\label{(2.9)}
\pppa S(t)a \in C((0,T];L^2(\OOO)).            
\end{equation}

By \eqref{(2.3)} with $\gamma=0$, we can easily verify that 
\begin{align*}
& \left\Vert \int^t_0 K(t-s)G(s) ds\right\Vert
\le C\int^t_0 (t-s)^{\alpha-1}\Vert G(s)\Vert ds\\
\le& C\Vert G\Vert_{L^{\infty}(0,T;L^2(\OOO))} \frac{t^{\alpha}}{\alpha}
\longrightarrow 0.
\end{align*}
Hence, with $S(t)a \in C([0,T];L^2(\OOO))$, we see that 
$\displaystyle\lim_{t\to 0} \Vert u(t) - a\Vert = 0$.

Moreover by Theorem 2.2 (i) in \cite{SY}, we see
$$
\pppa \left( \int^t_0 K(t-s)G(s) ds \right) \in L^2(\OOO\times (0,T)).
$$
This with \eqref{(2.8)}, we obtain $\pppa u \in L^1(0,T;L^2(\OOO))$.

Now we will prove 
$$
\int^t_0 K(t-s)G(s) ds \in C((0,T];\hhhh).
$$
For arbitrarily fixed $0<\delta_0<\delta$, we set 
$$
v_{\delta_0}(t) = \int^{t-\delta_0}_0 A_0K(t-s)G(s) ds,\quad
t \ge\delta.
$$
By \eqref{(2.3)} we can see that $v_{\delta_0} \in C([\delta,T];L^2(\OOO))$.
For $\delta \le t \le T$, by \eqref{(2.3)} we estimate
\begin{align*}
& \left\Vert \int^t_0 A_0K(t-s)G(s) ds - v_{\delta_0}(t)\right\Vert
= \left\Vert \int^t_{t-\delta_0} A_0K(t-s)G(s) ds \right\Vert\\
=& \left\Vert \int^t_{t-\delta_0} A_0^{1-\ep}K(t-s)A_0^{\ep}G(s) ds 
\right\Vert
\le C\int^t_{t-\delta_0} (t-s)^{\alpha\ep-1} \Vert A_0^{\ep}G(s)\Vert ds\\
\le &C\Vert A_0^{\ep}G\Vert_{L^{\infty}(0,T;L^2(\OOO))}
\frac{\delta_0^{\alpha\ep}}{\alpha\ep}.
\end{align*}
Hence 
$$
v_{\delta_0} \longrightarrow \int^t_0 A_0K(t-s)G(s) ds
\quad \mbox{in $C([\delta,T];L^2(\OOO))$}
$$
as $\delta_0 \to 0$, and by 
$v_{\delta_0} \in C([\delta,T];L^2(\OOO))$, we conclude that 
$$
\int^t_0 K(t-s)G(s) ds \in C([\delta,T];\hhhh)
$$
for any $\delta > 0$, and then 
$$
\int^t_0 K(t-s)G(s) ds \in C((0,T];\hhhh).
$$
Consequently by \eqref{(2.7)}, we obtain $u \in C((0,T]; \hhhh)$.

Finally, by \eqref{(2.3)} we have
\begin{align*}
& \left\Vert A_0\int^t_0 K(t-s)G(s) ds \right\Vert
= \left\Vert \int^t_0 A_0^{1-\ep}K(t-s)A_0^{\ep}G(s) ds \right\Vert\\
\le& C\int^t_0 (t-s)^{\alpha\ep-1} \Vert A_0^{\ep}G(s)\Vert ds
\le C\Vert A^{\ep}_0G\Vert_{L^{\infty}(0,T;L^2(\OOO))}
\frac{t^{\alpha\ep}}{\alpha\ep}.
\end{align*}
With \eqref{(2.2)}, the proof of the part (iii) is complete.
Thus the proof of Lemma \ref{Lemma 1} is complete.
\end{proof}

Henceforth we set 
$$
Bv(x) = \sum_{j=1}^d b_j(x)\ppp_jv(x) + c(x)v(x), \quad 
v\in \DDDD(B) = \hhhh.
$$

Next by Lemma \ref{Lemma 1}, we can prove

\begin{lem}\label{Lemma 2}
Let $F \in L^{\infty}(0,T;\DDDD(A_0^{\ep}))$ with some $\ep > 0$
and $a \in L^2(\OOO)$.  Then the solution $u$ to \eqref{(1.1)} 
belongs to $$C((0,T];\hhhh)$$ and there exists a constant $C>0$ depending on 
$T$, such that  
$$
\Vert u(T)\Vert_{H^2(\OOO)} 
\le C(t^{-\alpha}\Vert a\Vert 
+ \Vert A_0^{\ep}F\Vert_{L^{\infty}(0,T;L^2(\OOO))}), \quad t>0.
$$
\end{lem}
\begin{proof}{(of Lemma \ref{Lemma 2}).}
Without loss of generality, we can assume that $0 < \ep < \frac{1}{4}$.
By Lemma \ref{Lemma 1}, we have 
\begin{equation}\label{(2.10)}
u(t) = S(t)a + \int^t_0 K(t-s)F(s) ds 
+ \int^t_0 K(t-s)Bu(s) ds.       
\end{equation}
By Gorenflo, Luchko and Yamamoto \cite{GLY} or Kubica, Ryszewska and
Yamamoto \cite{KRY}, we know that there exists a unique solution
$u \in C([0,T]; L^2(\OOO))$ to \eqref{(2.10)}.  
Applying $A_0$ to equation (2.10), we have
$$
A_0u(t)\!\! =\!\! A_0S(t)a
+ \!\!\int^t_0 \!\!\!\!A_0^{1-\ep}K(t-s)A_0^{\ep}F(s) ds 
+\!\! \int^t_0 \!\!\!\!A_0^{1-\ep}K(t-s)A_0^{\ep}Bu(s) ds.
$$
Then, applying Lemma \ref{Lemma 1} (i), we obtain
\begin{align*}
& \Vert u(t)\Vert_{H^2(\OOO)}
\le Ct^{-\alpha}\Vert a\Vert 
+ C\int^t_0 (t-s)^{\alpha\ep-1}ds \Vert A_0^{\ep}F\Vert
_{L^{\infty}(0,T;L^2(\OOO))}\\
+& C\int^t_0 (t-s)^{\alpha\ep-1}\Vert u(s)\Vert_{H^2(\OOO)} ds\\
\le &C(t^{-\alpha}\Vert a\Vert + \Vert A_0^{\ep}F\Vert
_{L^{\infty}(0,T;L^2(\OOO))})
+ C\int^t_0 (t-s)^{\alpha\ep-1}\Vert u(s)\Vert_{H^2(\OOO)} ds.
\end{align*}
Here we used the following: by $0<\ep< \frac{1}{4}$ we have
$\Vert A_0^{\ep}v\Vert \sim \Vert v\Vert_{H^{2\ep}(\OOO)}$
for $v\in \DDDD(A_0^{\ep}) = H^{2\ep}(\OOO)$ (e.g., Fujiwara \cite{F}), 
and so
$$
\Vert A_0^{\ep}Bu(s)\Vert 
\le C\Vert Bu(s)\Vert_{H^{2\ep}(\OOO)} \le C\Vert u(s)\Vert_{H^2(\OOO)}
$$
because $Bu(s) \in H^1(\OOO) \subset \DDDD(A_0^{\ep})$ by 
$u(s) \in \hhhh$.
The generalized Gronwall inequality (e.g., Henry \cite{H} or
Lemma A.2 in \cite{KRY}) yields
\begin{align*}
& \Vert u(t)\Vert_{H^2(\OOO)}
\le C(t^{-\alpha}\Vert a\Vert + \Vert A_0^{\ep}F\Vert
_{L^{\infty}(0,T;L^2(\OOO))})\\
+ & Ce^{Ct}\int^t_0 (t-s)^{\alpha\ep-1}
(s^{-\alpha}\Vert a\Vert + \Vert A_0^{\ep}F\Vert_{L^{\infty}(0,T;L^2(\OOO))})
ds\\
\le& C(t^{-\alpha}\Vert a\Vert 
+ \Vert A_0^{\ep}F\Vert_{L^{\infty}(0,T;L^2(\OOO))})\\
+& Ce^{Ct}\left( t^{\alpha\ep-\alpha}
\frac{\Gamma(\alpha\ep)\Gamma(1-\alpha)}
{\Gamma(1-\alpha+\alpha\ep)}\Vert a\Vert
+ \frac{t^{\alpha\ep}}{\alpha\ep} \Vert A_0^{\ep}F\Vert
_{L^{\infty}(0,T;L^2(\OOO))}\right).
\end{align*}
Consequently 
$$
\Vert u(t)\Vert_{H^2(\OOO)}
\le C(t^{-\alpha}\Vert a\Vert 
+ \Vert A_0^{\ep}F\Vert_{L^{\infty}(0,T;L^2(\OOO))}).
$$
Thus the proof of Lemma \ref{Lemma 2} is complete.
\end{proof}

Finally we know

\begin{lem}\label{Lemma 3}
For $T>0$, the operator $$S(T): L^2(\OOO) \longrightarrow \hhhh$$ is surjective 
and there exist constants $C_1, C_2 > 0$ such that 
$$
C_1\Vert S(T)a\Vert_{H^2(\OOO)} \le \Vert a\Vert 
\le C_2\Vert S(T)a\Vert_{H^2(\OOO)}.
$$
\end{lem}

Lemma \ref{Lemma 3} is proved as Theorem 4.1 in \cite{SY}, whose proof is based on 
the representation of $S(T)a$ by the eigenfunction expansion and the 
complete monotonicity of $\eeaa$ (e.g., Gorenflo and Mainardi \cite{GM},
Pollard \cite{Pol}).
\section{Proofs of Theorems \ref{Theorem 1} and \ref{Theorem 2}}\label{sec3}

\subsection{Proof of Theorem \ref{Theorem 1}}
In terms of the lower-order part $B$ of the elliptic operator $-A$,
we can rewrite \eqref{(1.1)} as
\begin{equation}\label{(3.1)}
\left\{ \begin{array}{rl}
& \pppa u(t) = -A_0u(t) + Bu(t), \quad t>0, \\
& u(0) = a, \\
& u(t) \in H^1_0(\OOO), \quad 0 < t < T.
\end{array}\right.
\end{equation}
By Lemma \ref{Lemma 1} (ii), we have
\begin{equation}\label{(3.2)}
b:= u_a(T) = S(T)a + \int^T_0 K(T-s)Bu_a(s) ds.     
\end{equation}
Here, by $u_a(t)$, we denote the solution to \eqref{(3.1)}.  Applying 
Lemma \ref{Lemma 3} to 
\eqref{(3.2)}, we obtain
\begin{equation}\label{(3.3)}
a = S(T)^{-1}b - S(T)^{-1}\int^T_0 K(T-s)Bu_a(s) ds
=: S(T)^{-1}b - La,                
\end{equation}
where 
\begin{equation}\label{(3.4)}
La = S(T)^{-1}\int^T_0 K(T-s)Bu_a(s) ds.                 
\end{equation}
{\bf First Step.}
We prove that 
$L:L^2(\OOO) \longrightarrow L^2(\OOO)$ is a compact operator.
We set 
$$
L_0a = \int^T_0 K(T-s)Bu_a(s) ds, \quad a \in L^2(\OOO).
$$
Then $La = S(T)^{-1}L_0a$.

We choose $0 < \delta_0 < \delta_1 < \frac{1}{4}$.  We will estimate
$\Vert A_0^{1+\delta_0}L_0a\Vert$.
We note that $A_0^{\gamma}K(t)a = K(t)A_0^{\gamma}a$ for 
$\gamma\ge 0$ and $a \in \DDDD(A_0^{\gamma})$, which can be 
directly verified.  By \eqref{(2.3)}, we have
\begin{align*}
\Vert A_0^{1+\delta_0}L_0a\Vert
&= \left\Vert\int^T_0 A_0^{1+\delta_0}K(T-s)Bu_a(s) ds\right\Vert\\
&= \left\Vert \int^T_0 A_0^{1+\delta_0-\delta_1}K(T-s)A_0^{\delta_1}
B(u_a(s)) ds \right\Vert\\
&\le  C\int^T_0 (T-s)^{\alpha(\delta_1-\delta_0)-1}
\Vert Bu_a(s)\Vert_{H^1(\OOO)}ds\\
&\le C\int^T_0 (T-s)^{\alpha(\delta_1 - \delta_0)-1}
s^{-\alpha}\Vert a\Vert ds.
\end{align*}
For the last inequality, we used 
$0 < \delta_0 < \delta_1 < \frac{1}{4}$,
and $b_j, c \in C^1(\ooo{\OOO})$ and Lemma \ref{Lemma 2}, and  
$\DDDD(A_0^{\delta_1}) = H^{2\delta_1}(\OOO)$ (e.g., \cite{F}) and
\begin{align*}
& \Vert A_0^{\delta_1}Bu_a(s)\Vert 
\le C\Vert Bu_a(s)\Vert_{H^{2\delta_1}(\OOO)} \\
\le& C\Vert u_a(s)\Vert_{H^{1+2\delta_1}(\OOO)}
\le C\Vert A_0u_a(s)\Vert \le Cs^{-\alpha}\Vert a\Vert.
\end{align*} 
Therefore 
\begin{align*}
& \Vert A_0^{1+\delta_0}L_0a\Vert
\le C\Vert a\Vert\int^T_0 (T-s)^{\alpha(\delta_1-\delta_0)-1}
s^{-\alpha}ds\\
=& CT^{\alpha(\delta_1-\delta_0-1)}
\frac{\Gamma(\alpha(\delta_1-\delta_0))\Gamma(1-\alpha)}
{\Gamma(1 - \alpha + \alpha(\delta_1-\delta_0))}\Vert a\Vert
\end{align*}
because $\delta_1 - \delta_0 > 0$.

Since $\DDDD(A_0^{1+\delta_0}) \subset H^{2+2\delta_0}(\OOO)$ and the 
embedding\\ $H^{2+2\delta_0}(\OOO) \longrightarrow H^2(\OOO)$ is compact,
the operator $L_0: L^2(\OOO) \longrightarrow H^2(\OOO)$ is compact.
Moreover $S(T)^{-1}: H^2(\OOO) \longrightarrow L^2(\OOO)$ is bounded by 
Lemma \ref{Lemma 3}, we see that $L=S(T)^{-1}L_0: L^2(\OOO) \longrightarrow L^2(\OOO)$ 
is a compact operator. \\
{\bf Second Step.}
Since $b \in \hhhh$, by Lemma \ref{Lemma 3} we have
$p:= S(T)^{-1}b \in L^2(\OOO)$ and we rewrite \eqref{(3.3)} as
\begin{equation}\label{(3.5)}
(1+L)a = p \quad \mbox{in $L^2(\OOO)$}.    
\end{equation}
In the First Step, we already prove that $L:L^2(\OOO) \longrightarrow
L^2(\OOO)$ is compact.  Hence if we will prove that 
\begin{equation}\label{(3.6)}
La = -a \quad \mbox{implies} \quad a=0,              
\end{equation}
then the Fredholm alternative yields that $ (1+L)^{-1}: L^2(\OOO) 
\longrightarrow L^2(\OOO)$ is a bounded operator, and the proof can be
finished.

Equation \eqref{(3.6)} implies
$$
S(T)a + \int^T_0 K(T-s)Bu_a(s) ds = 0 \quad \mbox{in $L^2(\OOO)$}.
$$
Then we have to prove $a=0$.
For it, by means of Lemma \ref{Lemma 1} (ii), it is sufficient to prove
that if $w$ satisfies
$$\left\{ \begin{array}{rl}
& \pppa w(t) = -Aw(t),    \\
& w(t) \in H^1_0(\OOO), \quad 0 < t < T       
\end{array}\right.
$$
and
$w(T) = 0$ in $L^2(\OOO)$, then $w(0) = 0$.

We recall that the operator $A$ is defined by \eqref{(1.2)} with 
$\DDDD(A) = \hhhh$.  Then it is known that the spectrum 
$\sigma(A)$ of $A$ consists entirely of eigenvalues with 
finite multiplicities.  We denote $\sigma(A)$ by 
$\{\mu_1, \mu_2, ...\}$.   Here $\sigma(A)$ is a set and so 
$\mu_i$ and $\mu_j$, $i\ne j$ are mutually distinct.
Let $P_n$ be the projection for $\mu_n$, $n \in \N$ which is defined by
$$
P_n = \frac{1}{2\pi\sqrt{-1}} \int_{\gamma(\mu_n)}
(z-A)^{-1} dz,
$$
where $\gamma(\mu_n)$ is a circle centered at $\mu_n$ with sufficiently 
small radius such that the disc bounded by $\gamma(\mu_n)$ does not
contain any points in $\sigma(A)\setminus \{\mu_n\}$,
Then $P_n:L^2(\OOO) \longrightarrow L^2(\OOO)$ is a bouned linear operator
and $P_n^2 = P_n$ for $n\in \N$ (e.g., Kato \cite{Ka}).
Setting $m_n:= \mbox{dim}\, P_nL^2(\OOO)$, we have 
$m_n<\infty$.

The following is a fundamental fact.

\begin{lem}\label{Lemma 4}
If $y \in L^2(\OOO)$ satisfies $P_ny = 0$ for all $n \in \N$, 
then $y=0$.
\end{lem}
\begin{proof}
First we note 
$$
-(A^*v)(x) = \sumij \ppp_i(a_{ij}\ppp_jv)
- \sum_{j=1}^d \ppp_j(b_jv) + c(x)v, \quad 
\DDDD(A^*) = \hhhh,
$$
where $A^*$ is the adjoint operator of $A$.
Let $P_n^*$ be the adjoint operator of $P_n$:
$(P_n\va, \psi) = (\va, P_n^*\psi)$ for each $\va, \psi \in L^2(\OOO)$.

Then it is known (e.g., \cite{Ka}) that $\sigma(A^*) 
= \{ \ooo{\mu_n}\}_{n\in \N}$, where $\ooo{\mu}$ denotes the 
complex conjugate of $\mu \in \C$ and $P_n^*$ is the projection for 
the eigenvalue $\ooo{\mu_n}$ of $A^*$, and 
dim $P_n^*L^2(\OOO) = \mbox{dim}\, P_nL^2(\OOO) = m_n$.
Then by Theorem 16.5 in Agmon \cite{Ag}, we have
$$
\ooo{\mbox{Span}_{n\in \N}\, P_n^*L^2(\OOO)} = L^2(\OOO),
$$
that is, 
\begin{equation}\label{(3.7)}
(y, P_n^*\psi) = 0, \quad n\in \N, \, \psi\in L^2(\OOO) \quad 
\mbox{imply $y=0$}.                   
\end{equation}
Now we can complete the proof of Lemma 3.1.  Let $P_ny = 0$ for 
$n \in \N$.  Then $(P_ny, \psi) = 0$ for all $\psi \in L^2(\OOO)$.
Therefore $0 = (P_ny, \psi) = (y, P_n^*\psi)$ 
for all $n \in \N$ and $\psi \in 
L^2(\OOO)$, which yields $y=0$ by \eqref{(3.7)}.
\end{proof}
\noindent{\bf Third Step: completion of the proof of Theorem \ref{Theorem 1}.} Let we note $\pppa(P_nu(t)) = P_n\pppa u(t)$ because 
$P_n: L^2(\OOO) \longrightarrow L^2(\OOO)$ is a bounded operator.
We set $u_n(t) = P_nu(t)$.  Then 
$$
P_nAu_n(t) = Au_n(t) = -\mu_nu_n(t) + D_nu_n(t),
$$
where $D_n$ is an operator satisfying $D_n^{m_n} = O$, which 
corresponds to the Jordan canonical form.  Then \eqref{(3.1)} yields
$$\left\{ \begin{array}{rl}
& \pppa u_n(t) = (-\mu_n + D_n)u_n(t), \\
& u_n(0) = P_na, \qquad n \in \N.
\end{array}\right.
$$
We can define an operator $E_{\alpha,1}((-\mu_n + D_n)t^{\alpha})$ by the 
power series:
$$
E_{\alpha,1}((-\mu_n + D_n)t^{\alpha}) = \sum_{k=0}^{\infty}
\frac{(-\mu_n+D_n)^k t^{\alpha k}}{\Gamma(\alpha k + 1)}, \quad t>0.
$$
Then we can directly verify 
\begin{equation}\label{(3.8)}
u_n(t) = E_{\alpha,1}((-\mu_n + D_n)t^{\alpha})P_na, \quad t>0.
\end{equation}

Now we calculate the right-hand side of \eqref{(3.8)}.
Correspondingly to the Jordan canonical form, we can choose a suitable 
basis of $P_nL^2(\OOO)$:
$$
\psi_j^k: k=1,..., \ell_n, \quad j=1, ..., d_k
$$
satisfying $\sum_{k=1}^{\ell_n} d_k = m_n$, and
$$
\left\{ \begin{array}{rl}
& (A-\mu_n)\psi_1^k = 0, \\
& (A-\mu_n)\psi_2^k = \psi_1^k, \\
&\cdots \cdots \cdots, \\
& (A-\mu_n)\psi_{d_k}^k = \psi_{d_k-1}^k, \quad 1\le k \le \ell_n.
\end{array}\right.
$$
We expand $P_na$ in terms of this basis in $P_nL^2(\OOO)$:
$$
P_na = \sum_{k=1}^{\ell_n}\sum_{j=1}^{d_k} a^k_j \psi_j^k.
$$
Then 
\begin{align*}
& E_{\alpha,1}((-\mu_n + D_n)t^{\alpha})
(\psi_1^k \, \psi_2^k \, \cdots \, \psi_{d_k}^k)
\left(
\begin{array}{cc}
a_1^k \\
\vdots \\
a_{d_k}^k \\
\end{array}\right)\\
= &\sum_{m=0}^{\infty} t^{\alpha m}\frac{(-\mu_n+D_n)^m}{\Gamma(\alpha m + 1)}
(\psi_1^k \, \psi_2^k \, \cdots \, \psi_{d_k}^k)
\left(
\begin{array}{cc}
a_1^k \\
\vdots \\
a_{d_k}^k \\
\end{array}\right)\\
=& (\psi_1^k \, \psi_2^k \, \cdots \, \psi_{d_k}^k)\\
&\times\sum_{m=0}^{\infty} \frac{t^{\alpha m}}{\Gamma(\alpha m + 1)}
\left(
\begin{array}{ccccc}
-\mu_n^m & *        & \cdots & * & *\\
0        & -\mu_n^m & \cdots & * & *\\
\cdots   & \cdots   & \cdots & \cdots & \cdots\\
0        &  0       & \cdots & -\mu_n^m & *\\
0        & 0        & \cdots & 0        & -\mu_n^m \\ 
\end{array}\right)
\left(
\begin{array}{cc}
a_1^k \\
\vdots \\
a_{d_k}^k \\
\end{array}\right).
\end{align*}
Since $u_n(T) = 0$, we see that each component of the above is equal to
0 at $t=T$, and so 
\begin{equation}\label{(3.9)}
\left\{\begin{array}{rl}
& E_{\alpha,1}(-\mu_nT^{\alpha})a_1^k + \sum_{p=2}^{d_k} \theta_{1p}a_p^k
= 0, \\
& E_{\alpha,1}(-\mu_nT^{\alpha})a_2^k + \sum_{p=3}^{d_k} \theta_{2p}a_p^k
= 0, \\
& \cdots \cdots \cdots \\
& E_{\alpha,1}(-\mu_nT^{\alpha})a_{d_k-1}^k + \theta_{d_k-1, d_k}a_{d_k}^k
= 0,\\
& E_{\alpha,1}(-\mu_nT^{\alpha})a_{d_k}^k = 0,
\end{array}\right.
\end{equation}
where $\theta_{jp}$ with $j+1\le p \le d_k$ and 
$j=1,..., d_k-1$, are some constants depending also on $T$.
By the complete monotonicity (e.g., Gorenflo and Mainardi \cite{GM},
and Pollard \cite{Pol}), we see that 
$E_{\alpha,1}(-\mu_nT^{\alpha})\ne 0$.  Therefore by the backward 
substitution in (3.9), we can sequentially obtain 
$a_{d_k}^k=0$, $a_{d_k-1}^k=0$, ...., $a_1^k=0$ for $k=1,..., \ell_n$.
Hence $P_na=0$ for each $n \in \N$.  Then we reach $a=0$ in $L^2(\OOO)$.
Thus the proof of Theorem \ref{Theorem 1} is complete.
\qed

\subsection{Proof of Theorem 1.3}

Let $w = w(t)$ be the solution to 
$$\left\{ \begin{array}{rl}
& \pppa w(t) = -Aw(t) + F, \quad t>0,\\
& w(0) = 0, \quad w(t) \in H^1_0(\OOO), \quad t>0.
\end{array}\right.
$$
Since $F\in L^{\infty}(0,T;\DDDD(A_0^{\ep}))$, Lemma \ref{Lemma 2} proves
that $w \in C((0,T];\hhhh)$.  We consider
\begin{equation}\label{(3.10)}
\left\{ \begin{array}{rl}
& \pppa v(t) = -Av(t), \quad t>0,\\
& v(T) = b - w(T), \quad v(t) \in H^1_0(\OOO), \quad t>0.
\end{array}\right.                 
\end{equation}
By Theorem \ref{Theorem 1}, for $b \in \hhhh$, there exists a unique solution 
$v \in C([0,T];L^2(\OOO)) \cap C((0,T];\hhhh)$ such that 
$\pppa v \in C((0,T];L^2(\OOO))$ to \eqref{(3.10)}.
Setting $u=v+w$, we see that $u(T) = b - w(T) + w(T) = b$.
Then we can verify that $u$ satisfies
$$\left\{ \begin{array}{rl}
& \pppa u(t) = -Au(t) + F(t), \quad t>0,\\
& u(T) = b, \quad u(t) \in H^1_0(\OOO), \quad t>0.
\end{array}\right.
$$
The uniqueness of $u$ is seen by Theorem \ref{Theorem 1}.
Thus the proof of Theorem \ref{Theorem 2} is complete.
\begin{flushright}
$\square$
\end{flushright}

In future projects we would investigate similar pro\-blems where the 
principal part is an elliptic operator of order greater than $2$, 
like in \cite{FR}, and in the case of applied systems like \cite{CFGY}. 
Moreover we would study related inverse problems similarly to \cite{CFGY},
\cite{CFY1} and \cite{KCPF}.

\section*{Acknowledgment}
{\smaller\smaller\smaller The authors are indebted to the anonymous referee for the criticism and the useful suggestions which have made this paper easier 
to read and understand.\\
This work is also supported by the Istituto Nazionale di Alta Matematica (IN$\delta$AM),
through the GNAMPA Research Project 2019. 
Moreover, this research was performed in the framework of the 
French-German-Italian Laboratoire International Associ\'e (LIA), named COPDESC, on Applied Analysis,  issued by CNRS, MPI and IN$\delta$AM.}


\end{document}